\documentclass [12pt,a4paper,twoside,openright]{amsart}

\usepackage{mathtools}
\usepackage{amsmath}
\usepackage{tikz}
\usepackage{stix}
\usepackage{verbatim}
\usepackage[top=1in, bottom=1.25in, left=1.25in, right=1.25in]{geometry}
\usepackage{enumitem}
\usepackage{amsthm}
\usepackage{abstract}
\usepackage{graphicx}
\usepackage{mathtools}
\usepackage{dsfont}
\usepackage{bbm}
\newcommand{\measurerestr}{%
  \,\raisebox{-.127ex}{\reflectbox{\rotatebox[origin=br]{-90}{$\lnot$}}}\,%
}
\newcommand{\M}{\mathbb{M}}
\newcommand{\F}{\mathbb{F}}
\newcommand{\R}{\mathbb{R}}
\newcommand{\Le}{\mathscr{L}}

\newcommand{\T}{\mathbb{T}}
\newcommand*\circled[1]{\tikz[baseline=(char.base)]{
            \node[shape=circle,draw,inner sep=2pt] (char) {#1};}}

\newtheorem{definition}{Definition}[section]
\newtheorem*{theorem*}{Theorem}
\newtheorem{theorem}{Theorem}[section]

\newtheorem{lemma}{Lemma}[section]
\newtheorem{remark}{Remark}

\numberwithin{equation}{section}
\AtBeginDocument{\let\div\undefined\DeclareMathOperator{\div}{div}}
\providecommand{\abstract}{}
\date{}

\title[A current based approach for the continuity equation]{A current based approach for the uniqueness of the continuity equation}
\author{Tommaso Cortopassi}
\thanks{Scuola Normale Superiore, 56126 Pisa, Italy. E-mail: tommaso.cortopassi@sns.it}

\begin{document}

\maketitle

\begin{abstract}
     \noindent We consider the problem of proving uniqueness of the solution of the continuity equation with a vector field $u \in [L^1 (0,T; W^{1,p}(\T^d)) \cap L^\infty ((0,T) \times \T^d)]^d$ with $\div(u) ^- \in L^1 (0,T; L^\infty (\T^d))$ and an initial datum $\rho_0 \in L^q (\T^d)$, where $\T^d$ is the $d$-dimensional torus and $ 1 \leq p,q \leq +\infty$ such that $1/p + 1/q =1$ without using the theory of renormalized solutions introduced in \cite{diperna1989ordinary}. We propose a more geometric approach which will however still rely on a strong $L^1$ estimate on the commutator (which ultimately is the key technichal tool in \cite{diperna1989ordinary}, too), but other than that will be based on the theory of currents. 
\end{abstract}

\section{Introduction}
\noindent The aim of this work is to study the uniqueness of solutions of the continuity equation with a vector field $u \in [L^1 (0,T; W^{1,p}(\T^d)) \cap L^\infty ((0,T) \times \T^d)]^d$ and such that $ \div(u)^- \in L^1 (0,T ; L^\infty (\T^d))$, where $\T^d$ is the $d$-dimensional torus. This problem has already been extensively studied, starting from the seminal paper \cite{diperna1989ordinary} and later extended in \cite{ambrosio2004transport}, and the fundamental notion in their approach is that of ``renormalized solution". Another classical solution to the problem is obtained by using the duality method with the transport equation \cite[Lecture XVI]{ambrosio2021lectures}, where a more geometric point of view is also employed since a key step of the proof considers the pull back via the flow of the velocity field. We propose a different approach that relies on a more geometric argument, namely using the theory of currents, that will however still rely crucially on a commutator estimate which is the main technical point also in \cite{diperna1989ordinary} and \cite{ambrosio2021lectures}. We point out that recently a link between currents and the continuity equation has also been proposed in \cite{bonicatto2024existence} where they prove the well-posedness of a generalisation to $k$-currents of the continuity and of the transport equation in $\R^d$ (which are obtained in the cases $k=0$ and $k=d$ respectively). For these reasons, we believe that a more geometric point of view on the subject may be of interest for future developments. Consider the continuity equation
 
\begin{equation}\label{continuity equation}
\tag{PDE}
 \begin{cases}
\partial_t \rho (t,x) + \div_x (u(t,x) \rho (t,x))=0 &\text{ in } (0,1) \times \mathbb{T}^d\\
\rho(0,x)= \rho_0 (x) &\text{ in }  \mathbb{T}^d 
 \end{cases}
 \end{equation}
with $\rho_0  \in L^q (\mathbb{T}^d)$ and $u \in [L^1 (0,1; W^{1,p} (\T ^d )) \cap L^ \infty ((0,1) \times \T^d)]^d$, $1\leq p,q \leq + \infty$ with $1/p + 1/q = 1$. We assume without loss of generality that the final time is $T=1$. This is clearly not restrictive and it will avoid ambiguity in the following, since the letter $T$ will be reserved for another object, namely a normal 1-current representing the solution of \eqref{continuity equation}. We will restrict to the Sobolev regularity for simplicity, but as noted in Remark 5 the same argument can be made to work in the BV case, too. First of all we define the notion of solution for \eqref{continuity equation}.
\newline
\begin{definition}{\textbf{Solution of the continuity equation}}

    We say that a function $\rho \in L^\infty (0,1; L^q (\T^d))$ is a (distributional) solution to \eqref{continuity equation}, with $u \in L^1 (0,1; W^{1,p} (\T^d))$ and $\rho_0 \in L^q (\T^d)$, if

    $$ \int_0 ^1 \int_{\T^d} \rho( \partial_t \xi   + \langle u , \nabla  \xi \rangle) dx dt +  \int_{\T^d} \xi(0 , \cdot ) \rho_0 (x) dx   =0 \quad \forall \xi \in C^{\infty} _c ([0,1) \times \T^d) .$$
\end{definition}

\noindent
The starting point of our approach is to notice that $\rho$ solving $\eqref{continuity equation}$ can be written in an equivalent formulation with currents. In order to keep this work as self contained as possible, we provide in the Appendix the definition of $1$-current and of all the relevant tools and terminology necessary for the proof. For an extensive treatment on the subject we point the interested reader to \cite{federer2014geometric} and \cite{simon1983lectures}. 
 
 \begin{definition}{\textbf{Solution in the sense of currents}}\label{definition solution sense of currents}
 
A normal $1$-current $T$ in $(- \infty, 1) \times \T^d$ and supported in $[0,1) \times \T^d$ is said to be solution in the sense of currents of \eqref{continuity equation} if $\exists \rho \in L^{\infty} ( 0,1 ; L^q ( \T^d) ) $ such that 

\begin{equation}\label{continuity equation with current}
\tag{C-PDE}
 T = \rho e_t + \rho u^i e_i \text{ and } \partial T= -\rho_0 \mathcal{H}^{d} \measurerestr (\{0\} \times \T ^d) ,
 \end{equation}

where the boundary is considered with respect to $(- \infty, 1) \times \T^d$; $e_t$ is the unitary vector in the time direction; $e_1, e_2, \dots, e_d$ are unitary vectors in the $x_i$ directions and the $u^i$'s are the components of $u$.
\end{definition}

\begin{remark}

    Notice that we will be always interested in currents in $(- \infty, 1) \times \T^d$ and supported in $[0,1) \times \T^d$. The reason is that the definition of the boundary of a current is better given on open sets, and this is why  we only consider test functions in $C^\infty _c ([0,1) \times \T^d)$.
    \end{remark}

    \begin{remark}\label{remark consistent defnition of current solution}
    Let us check that the above definition is consistent. We have:

    \begin{align*}
    \langle \partial T, \xi\rangle = \langle T, d \xi \rangle = \int_0 ^1 \int_{\T^d} \rho [\partial_t \xi  + u \cdot \nabla _x \xi ] dx dt \qquad \forall \xi \in C_c ^\infty ([0,1) \times \T^d).
    \end{align*}

    Integrating by parts the first addendum:

    \begin{align*}
        \int_0 ^1 \int_{\T^d} \rho \partial _t \xi dx dt&= \int_{\T^d}\left( [\rho \xi ]_0 ^1 - \int_0 ^1 \xi \partial_t \rho dt\right) dx \\
        &= -\int_{\T^d} \rho_0 \xi dx - \int_0 ^1 \int_{\T^d} \xi \partial_t \rho dx dt.
    \end{align*}

    As for the second addendum, we use that 

    $$ \div_x (\rho u \xi )= \div_x (\rho u) \xi + \rho u \cdot \nabla _x \xi \text{ in the sense of distributions,}$$

    so 

    $$ \int_0 ^1 \int_{\T^d} \rho u \cdot \nabla_x \xi dx dt = - \int_0 ^1 \int_{\T^d} \xi \div_x (u \rho) dx dt .$$

    Using that $\rho$ solves \eqref{continuity equation} by hypothesis, we easily see that

    $$ \langle \partial T, \xi \rangle = - \int_{\T^d} \rho_0 \xi dx = \langle -\rho_0 \mathcal{H}^{d} \measurerestr (\{0\} \times \T ^d) ,\xi\rangle.$$
\end{remark}
\begin{remark}\label{repeated indices sum remark}
    Everywhere in the paper, if not otherwise specified, repeated indices are summed. For example $u^i e_i \coloneqq  \sum_{i=1} ^{d} u^i e_i$.
\end{remark}

Since the main issue with \eqref{continuity equation} is the uniqueness, we will give for granted the existence part which is obtained by a standard procedure of approximation by mollification of the velocity field. Given $u$ a velocity field satisfying 

\begin{equation}\label{main hypotheses on u}
\begin{cases}
    u \in [L^1 (0,1; W^{1,p} (\T^d))\cap L^\infty ((0,1) \times \T^d)]^d \\
    \div(u)^- \in L^1 (0,1; L^\infty (\T^d)) ,    
\end{cases}
\end{equation}

we will prove uniqueness in the following class:

\begin{equation}\label{main hypotheses on rho}
\begin{cases} 
    \rho_0 \in L^q (\T^d)\\
    \rho \in L^\infty (0,1; L^q (\T^d)), \text{ where } 1 \leq p,q \leq + \infty \text{ and }\frac{1}{p} + \frac{1}{q} =1. \\
\end{cases}
\end{equation}

In general, given a 1-current of the form $T= f_t e_t + f_1 e_1 + \dots + f_d e_d$ with $f_t , f_1, \dots, f_d \in L^1 _{loc} ((0,1) \times \T^d)$ we will use the following terminology:

\begin{equation}\label{horizontal and vertical parts definition}
\begin{cases} 
T^v \coloneqq f_1 e_1 + \dots +f_d e_d  &\text{will be called the ``vertical" part of } T \\
T^h \coloneqq f_t e_t &\text{will be called the ``horizontal" part of }T.  
\end{cases}
\end{equation}

\begin{remark}
    We will work on the torus and with $u$ bounded in order to avoid technical complications derived from having unbounded domains or an unbounded velocity field. 
\end{remark}

Before starting to prove the necessary preliminary lemmas, let us give an idea of the strategy of the proof. The key point is that if we have a $1$-current $T= f_t e_t + f_i e_i$ with null boundary in $(- \infty , 1) \times \T^d$ and supported on $[0,1) \times \T^d$, we can estimate its flat norm with its ``vertical mass", i.e. with $\sum_{i=1} ^d ||f_i||_{L^1 ((0,1) \times \T^d)} $ (see Lemma \ref{flat norm estimate null boundary} for a more precise statement). Then the idea is that if we have a solution $T= \rho e_t +\rho u^i e_i $ of \eqref{continuity equation with current} we can change variables in such a way that the resulting current is arbitrarily ``straightened", i.e. its vertical mass can be made arbitrarily small, and defining this straightening will be the content of Lemma \ref{lemma di raddrizzamento}. This, together with strong $L^1$ estimates on the commutator (see Lemma \ref{lemma commutator estimate}) will imply that the flat norm of $T$ is arbitrarily small and thus $T=0$.

\section{Proof of uniqueness}

We will work considering a 1-current $T= \rho e_t + \rho u^i e_i$ satisfying \eqref{continuity equation with current} with null boundary in $(- \infty, 1) \times \T^d$ and $\rho, u$ satisfying \eqref{main hypotheses on u} and \eqref{main hypotheses on rho} respectively. This is equivalent to asking that $\rho$ solves \eqref{continuity equation} with a null initial datum. We want to prove that $\rho=0$, which by the linearity of the equation is equivalent to proving uniqueness. Let us start with a series of general representation lemmas for $1$-currents with null boundary.

\begin{lemma}\label{flat norm estimate null boundary}

Let $T= f_t e_t + f_j e_j$ be a normal $1$-current such that $\partial T=0$ in $(- \infty , 1) \times \T^d$, supported in $[0,1) \times \T^d$  and such that $f_t , f_j \in L^\infty (0,1 ; L^1 (\T^d))$ for every $j=1, \dots,d$. Then $T$ is the boundary of a 2-current $S$ in $(- \infty , 1) \times \T^d$ defined as

$$ S \coloneqq \sum_{j=1} ^d F_j (t,x) e_t \wedge e_j ,$$

where 
\begin{align*}
     F_j (t,x) \coloneqq \begin{cases}
         -\int_0 ^t f_j (s,x) ds &\text{if } t >0\\
         0 &\text{otherwise}
     \end{cases}.
\end{align*}
 As a consequence, it holds the flat norm estimate 

$$ \F(T) \leq  \M (T^v) $$

where $T^v = f_1 e_1 + \dots f_d e_d$ is the ``vertical part" of $T$, as denoted in \eqref{horizontal and vertical parts definition}.
\end{lemma}
\begin{proof}
    Fix $\varepsilon >0$ and consider the mollified functions 
    $$ f_t ^\varepsilon \coloneqq \theta_\varepsilon \ast f_t \text{ and } f_j ^\varepsilon \coloneqq \theta_\varepsilon \ast f_j $$

    where $\theta_\varepsilon (t,x) \coloneqq \varepsilon^{-d-1} \theta\left( t/\varepsilon, x/\varepsilon \right)$ and $\theta$ is a standard mollifier in $\R^{d+1}$. By Remark \ref{remark consistent defnition of current solution} it can be easily seen that the current $T_\varepsilon =f_t ^\varepsilon e_t + f_j ^\varepsilon e_j$ has null boundary in $(- \infty , 1- \varepsilon) \times \T^d$. We remark that, as pointed out in Remark \ref{repeated indices sum remark}, we use the convention that repeated indices are summed if not otherwise specified. We also define

    $$ S_\varepsilon  \coloneqq \sum_{j=1} ^d F_j ^\varepsilon (t,x) e_t \wedge e_j , \text{ where }
     F_j ^\varepsilon (t,x) \coloneqq \begin{cases}
         -\int_0 ^t f_j ^\varepsilon (s,x) ds &\text{if } t >0\\
         0 &\text{otherwise}
     \end{cases}.
 $$

    Let $\omega (t,x)= \tau (t,x) dt + \xi^j (t,x)  dx_j$ with $\tau, \xi^j \in C_c ^\infty ([0,1-\varepsilon) \times \T^d)$ for every $j =1, \dots, d$ be a $1$-form.  We want to show that

    $$ \partial S_\varepsilon (\omega)= T_\varepsilon (\omega) = \int_0 ^{1-\varepsilon} \int_{\T^d} f_t ^\varepsilon \tau +  f_j ^\varepsilon \xi^j dx dt \qquad \forall \tau , \xi^1 , \dots, \xi^d \in C_c ^\infty ([0,1- \varepsilon) \times \T^d).$$

    By definition, 

    \begin{align*} \partial S_\varepsilon (\omega)&=  S_\varepsilon (d \omega)= S_\varepsilon \left(  \partial_t \xi^j dt \wedge dx_j - \partial_j \tau dt \wedge dx_j \right)\\
    &= S_\varepsilon (\partial_t \xi^j dt \wedge dx_j) - S_\varepsilon ( \partial_j \tau dt \wedge dx_j ). 
    \end{align*}

    The first addendum on the right hand side is:

    \begin{align} \label{first addendum}
    &S_\varepsilon( \partial_t \xi^j dt \wedge dx_j) = \int_0 ^{1- \varepsilon} \int_{\T^d}  F_j ^\varepsilon (t,x) \partial_t \xi^j (t,x) dx dt= \int_{\T^d}\int_0 ^{1- \varepsilon}   F_j ^\varepsilon (t,x) \partial_t \xi^j (t,x) dt dx \\
    &=\int_{\T^d}  \left([F_j ^\varepsilon (t,x)  \xi^j (t,x)]_0 ^{1- \varepsilon}  - \int_0 ^{1- \varepsilon}  \partial_t F_j ^\varepsilon (t,x)  \xi^j (t,x) dt \right) dx \nonumber \\
    &=  \int_0 ^{1- \varepsilon} \int_{\T^d} f_j ^\varepsilon \xi^j dx dt, \nonumber
    \end{align}

    where we use the fact that $F_j ^\varepsilon (0, \cdot)\equiv \xi^j (1- \varepsilon, \cdot) \equiv 0$ for every $j$. As for the second addendum:
    
    \begin{align}\label{second addendum}
         & - S_\varepsilon ( \partial_j \tau dt \wedge dx_j )=-\int_0 ^{1- \varepsilon} \int_{\T^d} F_j ^\varepsilon (t,x) \partial_j \tau (t,x) dx dt\\
         &= -\int_0 ^{1- \varepsilon} \int_{\T^d} \left( - \int_0 ^t f_j ^\varepsilon (s,x) ds \right) \partial_j \tau (t,x) dx dt \nonumber \\
         &=\int_0 ^{1- \varepsilon} \int_{\T^d} \int_0 ^t f_j ^\varepsilon (s,x)  \partial_j \tau (t,x) ds dx dt \nonumber \\
         &= \int_0 ^{1- \varepsilon} \int_{\T^d} \left(- \int_0 ^t \sum_{j=1} ^{d}\partial_j  f_j ^\varepsilon (s,x) ds \right) \tau (t,x) dx dt. \nonumber 
         \end{align}

    Since $\partial T_\varepsilon =0$ in $(- \infty , 1- \varepsilon) \times \T^d$ we have that, in a pointwise sense, $\partial_t f_t ^\varepsilon + \sum_{j=1} ^d  \partial_j f_j ^\varepsilon =0$ in $(- \infty,1- \varepsilon) \times \T^d$. Then 

    \begin{align}\label{second addendum final version}
         & - S_\varepsilon ( \partial_j \tau dt \wedge dx_j )= \int_0 ^{1- \varepsilon} \int_{\T^d} \left( \int_0 ^t \partial_t f_t ^\varepsilon (s,x) ds\right) \tau (t,x) dx dt \\
         &= \int_0 ^{1- \varepsilon} \int_{\T^d} f_t ^\varepsilon (t,x) \tau(t,x) dx dt.\nonumber 
    \end{align}
    
   Combining \eqref{first addendum} and \eqref{second addendum final version} we conclude that $\partial S_\varepsilon = T_\varepsilon$, and the thesis easily follows by letting $\varepsilon \to 0$.
\end{proof}

\begin{remark}
    Notice that in \eqref{second addendum final version} we implicitly used the fact that $f_t ^\varepsilon (0, \cdot)=0$. This can be deduced arguing as in Remark \ref{remark consistent defnition of current solution}. If we follow the same steps with $T_\varepsilon = f_t ^\varepsilon e_t + f_j ^\varepsilon e_j$, where we consider $f_t ^\varepsilon$ in place of $\rho$ and $f_j ^\varepsilon$ in place of $\rho u^j$, we have that for every $\xi \in C^\infty _c ([0,1- \varepsilon) \times \T^d)$:

    $$0=\langle  \partial T_\varepsilon, \xi \rangle = -\int_0 ^{1- \varepsilon} \int_{\T^d} \left(\partial_t f_t ^\varepsilon + \sum_{j=1} ^{d} \partial_j f_j ^\varepsilon \right) \xi dx dt - \int_{\T^d} f_t ^\varepsilon (0, x) \xi(0,x) dx.$$

    Restricting to $\xi \in C^\infty _c ((0,1- \varepsilon) \times \T^d)$ we can see that the double integral on the right hand side is 0, and then also $f_t ^\varepsilon(0, \cdot)=0$.
\end{remark}

\begin{lemma}\label{construction of 1-horizontal current with given boundary}

Let $T$ be a normal 1-current such that $\partial T= g$ in $(- \infty , 1) \times \T^d$ with $g \in L^1 ((0,1) \times \T^d)$ , and with $ \operatorname{Supp}(T) \subseteq [0,1) \times \T^d$ . Then there exists a horizontal (i.e. a current completely in the $e_t$ direction) normal $1$-current $M$ such that $\partial M =g$ in $(- \infty , 1) \times \T^d$, with support in $[0,1) \times \T^d$ and satisfying

$$ \M (M) \leq ||g||_{L^1 ((0,1) \times \T^d)}.$$

Moreover, if $\partial T$ is supported in a set $[0,1) \times K$ for some $K$, so is $M$.
\end{lemma}

\begin{proof}
     Consider 

    $$ M(t,x) \coloneqq -G(t,x) e_t \; \text{ with } \;G(t,x)= \begin{cases} \int_0 ^t g(s,x) ds & \text{if } t>0 \\
    0 & \text{otherwise}
    \end{cases}.$$

     Given $\xi \in C^\infty _c ([0,1) \times \T^d)$:

    $$ \langle \partial M , \xi \rangle = \langle M , d \xi \rangle = \int_{\T^d} \left(\int_0 ^1  -G(t,x) \partial_t \xi (t,x) dt \right) dx.$$

    But 

    \begin{align*} 
    \int_0 ^1 -G(t,x) \partial_t \xi (t,x) dt &= \overbrace{[ -G(t,x)\xi(t,x)]_{0} ^{1}}^{=0} + \int_0 ^1 \partial_t G(t,x) \xi(t,x) dt\\
    &= \int_0 ^1 g (t,x) \xi(t,x) dt ,
    \end{align*}

    so $\partial G= g$. If $g$ is supported in $[0,1) \times K$, it is clear from the construction that $M$ is supported in $[0,1) \times K$ too, and the mass norm estimate is trivial by construction. 
\end{proof}

 Let 
   
   \begin{equation}\label{regularised velocity and density} 
   u_\delta (t, \cdot) \coloneqq  \eta_\delta \ast u (t, \cdot) \text{ and } \rho_\delta (t, \cdot)\coloneqq \eta_\delta \ast \rho (t, \cdot)
   \end{equation}

   with $\eta_\delta (x) \coloneqq \delta ^{-d} \eta (x/\delta)$ where $\eta$ is a standard mollifier in the space variables only. Now, consider the current 
   
   \begin{equation}\label{regularised current}
   T_\delta \coloneqq \rho_\delta e_t + \rho_\delta  u_\delta ^i e_i.
   \end{equation}

  As already anticipated, the following commutator estimate will be crucial. Notice that contrary to \cite{diperna1989ordinary} where they only regularise the density $\rho$, we regularise both $\rho$ and the velocity field $u$ since we will need to use the flow of $u_\delta$ (see Lemma \ref{lemma di raddrizzamento}). However, as it will be shown in the proof, mollifying also the velocity does not create any problem.
   
   \begin{lemma}\label{lemma commutator estimate}
    Let $\rho$ be a solution of \eqref{continuity equation} with velocity $u$ and $\rho_0 =0$, and let $u, \rho$ satisfy assumptions \eqref{main hypotheses on u} and \eqref{main hypotheses on rho}. Then it holds

       $$ \partial_t \rho_\delta+ \div(u_\delta \rho_\delta )+ r_\delta=0 , $$

       where $r_\delta \coloneqq \eta_\delta \ast \div(u\rho) -\div (u_\delta \rho_\delta) \to 0$ as $\delta \to 0$ strongly in $L^1 ((0,1) \times \T^d) $.
   \end{lemma}

   \begin{proof}

        By convolution of the identity $\partial_t \rho + \div (u \rho)= 0$ with $\eta_\delta$ we get

        $$ \partial_t \rho_\delta + \div(u_\delta \rho_\delta )+ r_\delta =0 ,$$

        where $r_\delta = \eta_\delta \ast \div(u \rho) - \div (u_\delta \rho_\delta)$ and $u_\delta$, $\rho_\delta$ are defined as in \eqref{regularised velocity and density}. We recall that in distributional sense

        $$ \div(u \rho)= \div(u) \rho + u \cdot \nabla \rho,$$

        so we can rewrite $r_\delta$ as 
        
        \begin{equation}\label{definition commutator} r_\delta = \underbrace{ \eta_\delta \ast (\div(u) \rho) - \div(u_\delta ) \rho_\delta}_{= \circled{1}} + \underbrace{ \eta_\delta \ast ( u \cdot \nabla \rho) - u_\delta \cdot \nabla \rho_\delta }_{= \circled{2}}.
        \end{equation}

        As for \circled{1}, notice that

        \begin{align*}
            \eta_\delta \ast (\div(u) \rho) -\div(u_\delta) \rho_\delta  &= \eta_\delta \ast (\div(u) \rho) -\div(u) \rho + \\ & +\div(u) \rho -\div(u)\rho_\delta   + \\
            &+ \div(u) \rho_\delta -[\eta_\delta \ast \div(u)] \rho_\delta,
        \end{align*}

        and each line converges to 0 (strong) in $L^1 ((0,1) \times \T^d)$ for $\delta \to 0$. For example, for the first line we just need to use that $\div(u) \rho \in L^1 ((0,1) \times \T^d)$ to conclude, and similarly for the second and third line. As for $\circled{2}$, notice that

        $$ \circled{2}=  \underbrace{\eta_\delta \ast (u \cdot \nabla \rho) -u \cdot \nabla \rho_\delta  }_{= \circled{3}} +\underbrace{ u \cdot \nabla \rho_\delta -u_\delta \cdot \nabla \rho_\delta }_{= \circled{4}}.$$

        For \circled{4}, we will prove that for almost every $t \in (0,1)$ it holds that 
        
        $$u (t, \cdot)\cdot \nabla \rho_\delta (t, \cdot) -u_\delta (t, \cdot) \cdot \nabla \rho_\delta (t, \cdot) \to 0 \text{ as } \delta \to 0$$

        strongly in $L^1 (\T^d)$. For this reason, to keep the notation lighter we will omit writing the time dependence in the following. It holds :

        \begin{align*}
           &(u(x)- u_\delta (x)) \cdot \nabla \rho_\delta (x)\\
           &= \left(\int_{B(0,1)} \eta(y) (u(x) - u(x + \delta y)) dy \right) \cdot \left(\int_{B(0,1)} \frac{\nabla \eta (z)}{\delta } \rho(x+ \delta z) dz \right)\nonumber \\
           &= \left(\int_{B(0,1)} \eta(y) \frac{u(x) - u(x + \delta y)}{\delta} dy \right) \cdot \left( \int_{B(0,1)}\nabla \eta (y) \rho(x+ \delta y) dy \right).\nonumber 
        \end{align*}
        Integrating in the variable $x$ we have:

        \begin{align}\label{product}
            &\int_{\T^d} |(u(x)- u_\delta (x)) \cdot \nabla \rho_\delta (x)| dx \\
             &=\int_{\T^d} \left| \left(\int_{B(0,1)} \eta(y) \frac{u(x) - u(x + \delta y)}{\delta} dy \right) \cdot \left( \int_{B(0,1)}\nabla \eta (y) \rho(x+ \delta y) dy \right) \right| dx \nonumber  \\
             &\leq  \int_{\T^d} \left| \int_{B(0,1)} \eta(y) \frac{u(x) - u(x + \delta y)}{\delta} dy  \right| \; \left|   \int_{B(0,1)}\nabla \eta (y) \rho(x+ \delta y) dy  \right| dx. \nonumber 
        \end{align}

        At this point:

        \begin{equation}\left| \int_{B(0,1)} \eta(y) \frac{u(x) - u(x + \delta y)}{\delta} dy  \right| \to \left| \int_{B(0,1)} \eta(y) \nabla u(x) y dy  \right| \text{ as } \delta \to 0 \text{ in } L^p (\T^d,dx)\end{equation}

        strongly, by the strong convergence of difference quotients for Sobolev functions, and 

        \begin{equation}\label{Lq convergence} \left|   \int_{B(0,1)}\nabla \eta (y) \rho(x+ \delta y) dy  \right| \to \left|\int_{B(0,1)} \nabla \eta(y) \rho(x) dy \right| \text{ as } \delta \to 0 \text{ in } L^q (\T^d, dx).\end{equation}

        Then the product on the right hand side of \eqref{product} converges in $L^1 (\T^d, dx)$ to 

        \begin{align} \label{convergence in L1}
       \left| \nabla u(x) \int_{B(0,1)} \eta(y) y dy  \right|\;\left| \rho (x) \int_{B(0,1)} \nabla \eta(y) dy \right|=0,
        \end{align}

        because $\int_{B(0,1)} \eta(y) y dy =0$ since $\eta$ is even. So $\circled{4} \to 0$ strongly in $L^1 ((0,1) \times \T^d)$ as $\delta \to 0$. Finally, for \circled{3} we can use \cite[Proposition 4.7]{ambrosio2014continuity} to conclude.
   \end{proof}

\begin{remark}
    If in the previous Lemma we had $p=1$ and $q= +\infty$, we would not have strong convergence in the $L^\infty$ norm of \eqref{Lq convergence}. However this is not an issue, since the only convergence that really matters is the one of the product on the right hand side of \eqref{product} to \eqref{convergence in L1} in $L^1 (\T^d, dx)$, which in the case $p=1, q = + \infty$ can be achieved simply by dominated convergence.
\end{remark}

   We will now prove some lemmas regarding the ``straightening" procedure for $T$. First of all we need the following general lemma on push-forward of currents:

\begin{lemma}\label{push forward vector valued}
Let $F: \R^d \to \R^d$ be an invertible $C^1$ map with $C^1$ inverse and let $T= f_1 e_1 +  \dots f_n e_n \coloneqq (f_1 , \dots, f_n)^T $ be a 1-current of locally bounded mass in $\R^d$ with $f_1 , \dots f_n \in L^1 _{loc} (\R^d)$. Then:

\begin{equation}\label{push forward representation of vector current} F_\# T= \left( DF \cdot \frac{(f_1, \dots, f_n) ^T }{\det(DF)} \right) \circ F^{-1},
\end{equation}

with $DF$ denoting the jacobian matrix of $F$. This is, in some sense, the vector-valued case of the well-known formula:

\begin{equation}\label{push-forward formula for measures} 
F_\# (\rho \Le^n )= \frac{\rho}{\det(DF)} \circ F^{-1}  \Le^n .
\end{equation}
\end{lemma}

\begin{proof}
    It is just a computation. Consider $\omega= h^i dx_i$ a 1-form (recall that repeated indices are summed). By definition of push-forward of a current

    \begin{align}
        \langle F_\# T, \omega \rangle = \langle T, F^\# \omega \rangle.
    \end{align}

    Now,

    \begin{align}
        F^{\#} \omega = h^i (F(x)) dF^i = \left( h^i (F(x)) \frac{\partial F^i}{\partial x_j} (x) \right)dx_j,
    \end{align}

    so

    \begin{align*}  
    \langle F_\# T, \omega \rangle &= \int_{\R^d} h^i (F(x)) \frac{\partial F^i}{\partial x_j} (x) f_j (x) dx\\
    &= \int_{\R^d} h^i (y) \underbrace{\frac{\partial F^i}{\partial x_j} (F^{-1} (y)) f_j (F^{-1} (y) ) |\det( F^{-1} (y))|}_{=F_\# T} dy, 
    \end{align*}

    where in the last equality we changed variables setting $y= F (x)$. Finally, noticing that $Id =D(F \circ F^{-1} (y))=DF (F^{-1} (y)) \cdot DF^{-1} (y)$ we have 
    
    $$\det(DF^{-1} (y))= \frac{1}{\det(DF(F^{-1} (y)))}$$ 
    
    by Binet and so we have the thesis.
\end{proof}

\begin{lemma}\label{lemma di raddrizzamento}
    Let
    
    $$T_\delta = \rho_\delta e_t + \rho_\delta  u_\delta ^i e_i \text{ be a } 1\text{-current in } \R^d $$
     and
     
    $$\psi_\delta (t,x) \coloneqq (t, \Phi_\delta (t,x)) ,$$
    
    with $u_\delta, \rho_\delta$ defined as in \eqref{regularised velocity and density} and $\Phi_\delta$ the flow of $u_\delta$. It holds that 

    $$ (\psi_\delta ^{-1} )_\# (T_\delta) =\left[ \left( \frac{ \rho_\delta}{\det(D_x \Phi_\delta ^{-1} )} \right) \circ \underbrace{(t, \Phi_\delta (t,x))}_{= \psi_\delta (t,x)} \right] e_t . $$
    
\end{lemma}

\begin{proof}

By Lemma \ref{push forward vector valued} it holds the representation \eqref{push forward representation of vector current} with $F= \psi_\delta ^{-1}$. We have:

$$ D\psi_\delta ^{-1} (\psi_\delta (t,x))= 
\left[ \begin{array}{c|c}
 1 & 0 \\
 \hline \\
    \partial_t \Phi_\delta ^{-1} (t, \Phi_\delta (t,x)) & D_x \Phi_\delta ^{-1} (t, \Phi_\delta (t,x))
\end{array}\right]
$$

and since $\Phi_\delta ^{-1} (t,\Phi_\delta(t,x))= x$ for every $(t,x)$ and $\partial_t \Phi_\delta (t,x)= u_\delta (t, \Phi_\delta (t,x))$:

\begin{align} 
0 &= \frac{d}{dt}( \Phi_\delta ^{-1} (t,\Phi_\delta(t,x))) = \partial_t \Phi_\delta ^{-1} (t,\Phi_\delta(t,x)) + D_x \Phi_\delta ^{-1} (t,\Phi_\delta(t,x)) \partial_t \Phi_\delta (t,x)=  \\
&=\partial_t \Phi_\delta ^{-1} (t,\Phi_\delta(t,x)) + D_x \Phi_\delta ^{-1} (t,\Phi_\delta(t,x)) u_\delta (t, \Phi_\delta (t,x)) .\nonumber 
\end{align}

Then we can rewrite:

$$D\psi_\delta ^{-1} (\psi_\delta (t,x))= 
\left[ \begin{array}{c|c}
 1 & 0 \\
 \hline \\
    0 & D_x \Phi_\delta ^{-1} (t, \Phi_\delta (t,x)) 
    \end{array} \right]
\left[ \begin{array}{c|c}
 1 & 0 \\
 \hline \\
    -u_\delta (t, \Phi_\delta (t,x)) & Id  
\end{array}\right] $$

and by Lemma \ref{push forward vector valued}:

\begin{align*} (\psi_\delta ^{-1} )_\# (T_\delta) &= \left( \frac{1}{\det(D_x \Phi_\delta ^{-1} )} \left[ \begin{array}{c|c}
 1 & 0 \\
 \hline \\
    0 & D_x \Phi_\delta ^{-1}  
    \end{array} \right]
\left[ \begin{array}{c|c}
 1 & 0 \\
 \hline \\
    -u_\delta  & Id  
\end{array}\right]  \left[ \begin{array}{c}
      \rho_\delta\\
      \rho_\delta u_\delta
\end{array}\right]\right) \underbrace{(t, \Phi_\delta (t,x))}_{= \psi_\delta (t,x)} = \\
&= \left[\left(\frac{\rho_\delta }{\det (D_x \Phi_\delta ^{-1})} \right) \circ (t, \Phi_\delta (t,x) \right]e_t.
\end{align*}

\end{proof}

We are now ready to conclude.

\begin{theorem}{\textbf{Main Result}}\label{uniqueness of current with 0 boundary}

   Let $T=(\rho , \rho u)$ be a 1-current solving \eqref{continuity equation with current} in the sense of Definition \ref{definition solution sense of currents} with $u ,\rho$ satisfying \eqref{main hypotheses on u}, \eqref{main hypotheses on rho} and $\rho_0 =0$. Then $T=0$.
\end{theorem}

\begin{proof}
Fix $\delta >0$ and consider $u_\delta, \rho_\delta$ as in \eqref{regularised velocity and density} and $T_\delta \coloneqq \rho_\delta e_t +  \rho_\delta u_\delta ^i e_i$. We have:

\begin{equation}\label{final splitting of T} T= \underbrace{(T-T_\delta + R_\delta)}_{\circled{1}} + \underbrace{T_{\delta}}_{\circled{2}}  - \underbrace{R_\delta}_{\circled{3}} .
\end{equation}

with $R_\delta$ constructed as in Lemma \ref{construction of 1-horizontal current with given boundary} from the boundary $\partial T_\delta = r_\delta$, where the boundary is considered in $(- \infty, 1) \times \T^d$. By construction $\circled{1}$  has null boundary in $(- \infty , 1) \times \T^d$ and by Lemma \ref{flat norm estimate null boundary} its flat norm can be estimated by its vertical mass, i.e.

    $$ \F(T- T_\delta - R_\delta) \leq ||\rho u - \rho_\delta u_\delta||_{L^1 ((0,1) \times \T^d)} \to 0 \text{ as } \delta \to 0.$$

As for $\circled{3}$, its mass norm can be estimated by $||r_\delta||_{L^1 ((0,1) \times \T^d)}$, which we know to tend to $0$ as $\delta \to 0$ by Lemma \ref{lemma commutator estimate}. So we only need to prove that 

$$\F(T_\delta) \to 0 \text{ as }\delta \to 0.$$

In order to do that, we consider the push-forward of $T_\delta$ via $\psi_\delta ^{-1}$, where we recall that $\psi_\delta (t,x)= (t, \Phi_\delta (t,x))$, and by Lemma \ref{lemma di raddrizzamento} we have

$$ (\psi_\delta ^{-1} )_\# T_\delta = \left[\left(\frac{\rho_\delta }{\det (D_x \Phi_\delta ^{-1})} \right) \circ (t, \Phi_\delta (t,x)) \right]e_t.$$

Consider

$$\tilde{r}_\delta \coloneqq \partial (\psi_\delta ^{-1} )_\#  T_\delta= (\psi_\delta ^{-1} )_\# \partial T_\delta = (\psi_\delta ^{-1} )_\# r_\delta  .$$

Using the representation formula for push-forward  measures (see \eqref{push-forward formula for measures}) we can easily see that $||\tilde{r}_\delta ||_{L^1 ((0,1) \times \T^d)}=||r_\delta ||_{L^1 ((0,1) \times \T^d)}$, indeed for every $t \in (0,1)$ fixed:

\begin{align}\label{r tilde norm equal to r norm}
    &\int_{\T^d} |\tilde{r}_\delta| (t,x) dx = \int_{\T^d} \frac{|r_\delta| (t, \Phi_\delta (t,x))}{\det (D_x \Phi_\delta ^{-1} (t, \Phi_\delta (t,x)))} dx = \\
    &=\int_{\T^d} \frac{|r_\delta| (t,y)}{\det (D_x \Phi_\delta ^{-1} (t, y)) \det (D_x \Phi_\delta  (t, \Phi_\delta ^{-1} (t,y)))} dy
    =\int_{\T^d} |r_\delta|(t,y) dy,\nonumber 
\end{align}

where we used the change of variables $y= \Phi_\delta (t,x)$ and the fact that

$$\det (D_x \Phi_\delta ^{-1} (t, y) )  \det (D_x \Phi_\delta  (t, \Phi_\delta ^{-1} (t,y)))= \det (\underbrace{D_x \Phi_\delta ^{-1} (t, y) \cdot  D_x \Phi_\delta  (t, \Phi_\delta ^{-1}(t,y))}_{= Id})\equiv 1 .$$

By Lemma \ref{construction of 1-horizontal current with given boundary}, $\exists P_\delta$ a horizontal $1$-current with mass controlled by $||\tilde{r}_\delta||_{L^1 ((0,1) \times \T^d)} = ||r_\delta||_{L^1 ((0,1) \times \T^d)}$ and such that $\partial P_\delta = \tilde{r}_\delta $. By construction $(\psi_\delta)^{-1} _\# T_\delta - P_\delta$ has null boundary in $(- \infty ,1) \times  \T^d$ and it is completely horizontal , so by Lemma \ref{flat norm estimate null boundary} we have: 

$$ \F ((\psi_\delta)^{-1} _\# T_\delta - P_\delta )= 0 \implies (\psi_\delta)^{-1} _\# T_\delta- P_\delta =0 \implies T_\delta = (\psi_\delta)_\# P_\delta.$$

Then:

$$\F(T_\delta) \leq \M ( T_\delta )= \M((\psi_\delta )_\# P_\delta).$$

We can now calculate $\M((\psi_\delta )_\# P_\delta)$. Consider:

\begin{itemize}
    \item A $1$-form $\omega (t,x)= \tau (t,x) dt + \xi ^j (t,x) dx_j$ with $\tau, \xi^j \in C_c ^\infty ([0,1) \times \T^d)$ such that $|\tau|, |\xi^j |\leq 1$ for every $j$;

    \item A function $f \in L^1 ((0,1) \times \T^d)$ such that $P_\delta (t,x)= f(t,x) dt$. We refer the reader to Lemma \ref{construction of 1-horizontal current with given boundary} for the precise definition of $f$, for which it holds $||f||_{L^1 ((0,1) \times \T^d)} \leq ||\tilde{r}_\delta||_{L^1 ((0,1) \times \T^d)} = ||r_\delta||_{L^1 ((0,1) \times \T^d)} $  (see \eqref{r tilde norm equal to r norm}) by construction.
\end{itemize}
Then:

\begin{align*}
(\psi_\delta )_\# P_\delta (\omega)= P_\delta ((\psi_\delta)^\# \omega)= P_\delta (\tau (\psi_\delta (t,x)) d \psi_\delta ^t + \xi^j (\psi_\delta (t,x)) d \psi_\delta ^j),
\end{align*}

where with $\psi_\delta ^t$ we denote the component of $\psi_\delta$ in the direction $e_t$ and with $\psi_\delta ^j$ the compontents of $\psi_\delta$ in the directions $e_j$. Since $P_\delta$ is completely horizontal, only the partial derivatives with respect to the time variable will matter for the computation, i.e.

\begin{align*} 
&P_\delta (\tau (\psi_\delta (t,x)) d \psi_\delta ^t + \xi^j (\psi_\delta (t,x)) d \psi_\delta ^j)\\
&= P_\delta \left(\left( \tau(\psi_\delta (t,x)) + \sum_{j=1} ^{d} u_\delta ^j (\psi_\delta (t,x)) \xi^j (\psi_\delta (t,x))  \right) dt \right)\\
&= \int_0 ^1 \int_{\T^d} \left( \tau(\psi_\delta (t,x))  + \sum_{j=1} ^{d} u_\delta ^j (\psi_\delta (t,x)) \xi^j (\psi_\delta (t,x))\right) f (t,x) dx dt,
\end{align*}

where we used that $\partial_t \psi_\delta ^j (t,x) = u^j (t, \Phi_\delta (t,x))= u^j (\psi_\delta (t,x))$ for every $j= 1, \dots, d$. But since $u \in L^\infty ((0,1) \times \T^d)$ and by assumption $|\tau| , |\xi ^j| \leq 1$ for every $j=1, \dots, d$ we have that

$$ \F(T_\delta) \leq \M(T_\delta) = \M( (\psi_\delta )_\# P_\delta ) \lesssim \M(P_\delta ) = ||r_\delta ||_{L^1 ((0,1) \times \T^d)} \to 0 \text{ as } \delta \to 0,$$

and we conclude since, by choosing $\delta$ small enough, we can make the flat norm of $T$ arbitrarily small.

\end{proof}

\begin{remark}
    The previous result can actually be extended to the case  of a velocity field $ u \in [L^1 (0,1; BV(\T^d)) \cap L^\infty ((0,1) \times \T^d)]^d$. The only difference will be in the proof of Lemma \ref{lemma commutator estimate} where the choice of a standard mollifier $\eta$ will not be enough to ensure strong convergence to 0 in $L^1 ((0,1) \times \T^d)$. This can be fixed by employing ad hoc local anisotropic mollifiers, although the procedure is more involved. For the details, see \cite{ambrosio2004transport}.
\end{remark}
\section{Appendix}

In order to keep this work as self contained as possible, we give here the necessary definitions. Since we mainly work with $1$-currents, we will only define them, but the interested reader may find more general definitions in \cite{federer2014geometric} or \cite{simon1983lectures}. In the following we will always assume $U, V \subset \R^d$ open sets.

\begin{definition}{\textbf{1-currents}}

    A $1$-current on $U$ is a continuous linear functional on the space $\mathcal{D}^1 (U)$ of $1$-differential forms on $U$. That is, $1$-currents are elements of the dual of $\mathcal{D}^1 (U)$, and this set is denoted as $\mathcal{D}_1 (U)$.
\end{definition}
\newpage
\begin{definition}{\textbf{Pull-back and push forward}}

    Let $\omega(x) = \sum_{1 \leq i \leq d} \xi_i(x)  dx_i $ be a $1$- differential form in $\mathcal{D}^1 (V)$, with $\xi_i \in C^\infty _c (V)$ for every $i$, and let $\phi : U \mapsto V$ be a smooth map. Then we denote the pull-back of $\omega$ via $\phi$ as $\phi^\# \omega \in \mathcal{D}^1 (U)$, which is defined as 
    
    $$\phi ^\# \omega (x)\coloneqq \sum_{1\leq i \leq d } \xi_i( \phi(x)) \;d \phi_i.$$

    Similarly, given a $1$- current $T \in \mathcal{D} _1 (U)$, we denote the push-forward of $T$ via $\phi$ as $\phi_\# T \in \mathcal{D}_1 (V)$, and it is defined as

    $$ \phi_\# T( \omega) \coloneqq T (\phi^\# \omega) \quad \text{ for all } \omega \in \mathcal{D} ^1 (V).$$
\end{definition}

\begin{definition}{\textbf{Boundary of a 1-current}}

Let $T \in \mathcal{D}_1 (U)$. We define the boundary of $T$ as a distribution on $U$ (i.e. an element of the dual space of $C^\infty _c (U)$) as

$$ \partial T (\xi ) \coloneqq T(d \xi) \quad \text{ for all } \xi \in C^\infty _c (U).$$
    
\end{definition}

\begin{remark}[Properties of push-forward and boundary, and of pull-back and differential]

It can be verified that the operations of pull-back and differential commute, i.e if we consider a differential form $\omega$ and a pull-back via a suitable map $\phi$, then

$$ \phi^\# (d\omega) = d (\phi^\# \omega).$$

Similarly, the boundary operator commutes with the push-forward for currents, i.e. for every current $T$ and for every suitable map $\phi$ it holds that

$$ \phi_\# (\partial T) = \partial (\phi_\# T) .$$
\end{remark}

\begin{definition}{\textbf{Mass of a 1-current}}

Let $T \in \mathcal{D}_1 (U)$. We define the mass of $T$ as 

$$ \M (T)= \sup_{\substack{\omega \in \mathcal{D}^1 (U) \text{ s.t.} \\|\xi_i |\leq 1 \text{ for every } i}} T(\omega), \quad \text{ with } \omega(x) = \sum_{1 \leq i \leq d} \xi_i (x) dx_i . $$

If $S$ is a distribution in $U$ , then we define its mass as

$$ \M(S) \coloneqq \sup_{ \substack{\xi \in C_c ^\infty (U) \text{ s.t.}\\ |\xi| \leq 1}} S(\xi).$$
    
\end{definition}

\begin{definition}{\textbf{Normal 1-currents and diffuse 1-currents}}

We say that a $1$-current $T \in \mathcal{D}_1 (U)$ is normal if 

$$ \M(T) + \M(\partial T) < + \infty.$$

Moreover, we say that a $1$-current $T$ is diffused if it can be represented as 

$$ T= \sum_{1 \leq i \leq d} f_i e_i,$$

 with $f_i \in L^1 _{loc} (U)$  for every $i$. In this case, the action of $T$ on a $1$-form \newline $\omega(x)= \sum_{1 \leq i \leq d} \xi_i (x) dx_i \in \mathcal{D}^1 (U)$ is

$$ T(\omega)= \int_U \left(\sum_{1 \leq i \leq d} f_i (x) \xi_i (x) \right) dx.$$
\end{definition}

\begin{definition}{\textbf{Flat norm}}

    Let $ T \in \mathcal{D}_1 (U)$. We define the flat norm of $T$ as 

    $$ \F(T) \coloneqq \inf \{\M (S) + \M (L) \text{ s.t. } T= \partial S + L\}.$$

    In particular, we point out that $S$ is a $2$-current (see \cite{federer2014geometric}, \cite{simon1983lectures}) and that $\F$ is a norm, so $\F(T)=0 \implies T=0$.
\end{definition}

\section{Acknowledgements}

The author wishes to thank prof. Giovanni Alberti and prof. Luigi Ambrosio for the constructive discussions and the precious help.

    \bibliographystyle{plain}
    \bibliography{bibliography.bib}

\end{document}